\documentclass[11pt]{article}
\usepackage{color} 
\usepackage{amssymb,amsmath,mathrsfs,amsthm}
\usepackage{enumitem}
\usepackage{mathtools}
\usepackage{geometry}
\usepackage{amsfonts}
\usepackage{color}

\usepackage{esint,comment}

\usepackage[dvips]{graphicx}
\allowdisplaybreaks
\newtheorem{thm}{Theorem}
\newtheorem{cor}[thm]{Corollary}
 
\newtheorem{lem}[thm]{Lemma}
 \newtheorem{lemma}[thm]{Lemma}

\theoremstyle{definition}
\newtheorem{defn}[thm]{Definition}

\def \no#1#2#3 {{\bf #1} (#3), #2.}
\def \eds#1#2#3 {#1, #2, #3.}

\def\d{{\rm d}}

\def\N{{\mathbb N}}

\def\:{{\colon}}

\def\be#1{\begin{equation}\label{#1}}
\def\ee{\end{equation}}

\def\<{\langle}
\def\>{\rangle}
\def\coloneqq{:=}

\newcommand{\na}{\nabla}

\renewcommand{\P}{\mathbb{P}}
\newcommand{\lec}{\lesssim}
\newcommand{\bs}{\begin{split}}
\newcommand{\essss}{\end{split}}

\newcommand{\M}{\mathcal{M}}

\renewcommand{\div}{\operatorname{div}}
 
\newcommand{\eqnb}{\begin{equation}}
\newcommand{\eqne}{\end{equation}}

\renewcommand{\ee}{\mathrm{e}}

\newcommand{\p}{\partial}

\newcommand{\cR}{\mathcal{R}}
\newcommand{\cC}{D}
\newcommand{\RR}{\mathbb{R}}
\newcommand{\PP}{\mathbb{P}}

\renewcommand{\d}{\mathrm{d}}

\newcommand\blfootnote[1]{%
  \begingroup
  \renewcommand\thefootnote{}\footnote{#1}%
  \addtocounter{footnote}{-1}%
  \endgroup
}

\begin{document}
\title{Quantitative transfer of regularity of the incompressible Navier-Stokes equations from $\RR^3$ to the case of a bounded domain}
\author{W. S. O\.za\'nski}
\date{\vspace{-5ex}}
\maketitle
\blfootnote{Department of Mathematics, University of Southern California, Los Angeles, CA 90089, USA, and Institute of Mathematics, Polish Academy of Sciences, Warsaw, 00-656, Poland; email: ozanski@usc.edu }

\begin{abstract}
Let $u_0\in C_0^5 ( B_{R_0})$ be divergence-free and suppose that $u$ is a strong solution of the three-dimensional incompressible Navier-Stokes equations on $[0,T]$ in the whole space $\RR^3$ such that $\| u \|_{L^\infty ((0,T);H^5 (\RR^3 ))} + \| u \|_{L^\infty ((0,T);W^{5,\infty }(\RR^3 ))}  \leq M <\infty$. We show that then there exists a unique strong solution $w$ to the problem posed on $B_R$ with the homogeneous Dirichlet boundary conditions, with the same initial data and on the same time interval for $R\geq \max(1+R_0, C(a)  C(M)^{1/a} \exp ({CM^4T/a}) )$ for any $a\in [0,3/2)$, and we give quantitative estimates on $u-w$ and the corresponding pressure functions.
\end{abstract}


\section{Introduction}
We are concerned with the three-dimensional incompressible Navier-Stokes equations 
\eqnb\label{NSE_intro}
\begin{split}
\p_t u - \nu \Delta u + (u\cdot \nabla ) u + \nabla \pi &=0,\\
\div u &=0
\end{split}
\eqne
in $\RR^3$, where $\nu>0$ is the viscosity coefficient, $u$ is the velocity of a fluid, and $p$ is the pressure function. The equations are equipped with an initial condition $u(0)=u_0$. The study of the equations goes back to the work of Leray \cite{leray_34} and Hopf \cite{hopf_1951}, who showed the global-in-time existence of weak solutions in the case of $\RR^3$ (Leray) and the case of a bounded, smooth domain $\Omega \subset \RR^3$ (Hopf). These are usually referred to as \emph{Leray-Hopf weak solutions}. We refer the reader to the recent comprehensive review article \cite{op} of Leray's work and to \cite{NSE_book} for a general background of the mathematical theory of the Navier-Stokes equations \eqref{NSE_intro}.
We note that the fundamental question of global-in-time existence of strong solutions remains open in each of these settings. 

Heywood \cite{heywood_88} was the first to study the connections between the Navier-Stokes equations posed on different domains, and he showed that a Leray-Hopf weak solution on $\RR^3$ can be obtained as a limit of weak solutions on $B_R$. In the two-dimensional case Kelliher~\cite{kelliher_08} proved that weak solutions on large domains converge strongly, in the energy space, to a weak solution on $\RR^2$, provided the latter exists on the same time interval. Some connections regarding existence of smooth solutions on in various setting were explored from a different point of view by Tao \cite{tao_2013}. 

A more direct link regarding well-posedness question between the case of the whole space and the torus has been recently shown by Robinson \cite{robinson_torus}, who used a compactness method to show that for a localized initial data, given the solution remains strong in $\RR^3$ until time $T$, the same is true in the case of sufficiently large periodic torus. This is the problem of ``transfer of regularity'' that we are concerned with in this note. It is closely related to numerical analysis of problems of fluid mechanics (see \cite{kerr_2018}, for example), where often an infinite domain must be approximated by a bounded domain.

Some related problems of transfer of regularity have been studied in the context of vanishing viscosity limits, where the existence of smooth solution to the Euler equations implies the existence of a smooth solution to the Navier-Stokes equations with a sufficiently small viscosity \cite{constantin_1986}. Another related phenomenon is that sufficiently smooth solution of the Navier-Stokes equations on a time interval $[0,T]$ gives similar regularity of some numerical schemes \cite{CCRT_07}. Some other works \cite{raugel_sell_93,gallagher_97} deduce regularity of 3D flows that are, in some sense, ``sufficiently two-dimensional.'' We refer the reader to \cite{robinson_torus} for a further discussion. We note that it can be verified that the considerations of Robinson \cite{robinson_torus} translate to the case of a bounded, smooth domain $\Omega \subset \RR^3$. This gives a transfer of regularity from $\RR^3$ to a sufficiently large $\Omega $. However, due to the use of the compactness method, it is not clear from this approach how large the approximating domain would need to be. We address this issue here.  

Our main theorem gives the first quantitative result regarding the size of the domain $\Omega$ on which the problem \eqref{NSE_intro}, equipped with the homogeneous Dirichlet boundary conditions, has a unique strong solution on the same time interval and with the same initial data.   

\begin{thm}[Main result]\label{thm_main}
Let $R_0>0$ and $M\geq 1$, $a\in [0,3/2)$ and assume that $u_0 \in C_0^5 (B_{R_0})$ is divergence free. Suppose that $(u,\pi)$ is a strong solution of \eqref{NSE_intro} with the initial condition $u(0)=u_0$ on $\RR^3$ for $t\in [0,T]$, such that 
\[ \|  u (t) \|_{H^5}+ \| u (t) \|_{W^{5,\infty}} \leq  M\]  for all $t\in [0,T]$. Then for every $R\geq R_0+1$ such that
\eqnb\label{how_large_R}
 R\gtrsim_{a,M,u_0 }   \ee^{C M^4 T/a },
\eqne 
where $a\in [0,3/2)$, there exists a unique strong solution $(w,\tilde{\pi })$ to the problem \eqref{NSE_intro} posed on $B_R$ with the homogeneous boundary condition, $ \left. u\right|_{\partial B_R} =0$, and the same initial condition $w(0) = u_0$ (see Definition~\ref{def_strong_sol} below), where $C>1$ is a universal constant. Moreover
\eqnb\label{main_ests}
\begin{split}
\| \nabla (u -w ) (t) \|_{L^2(B_R)} &\lec_{a,M,u_0 } \ee^{CM^4t} R^{-a} \\
 \text{ and }\quad \| \nabla (\pi - \tilde{\pi } ) \|_{L^p((0,t);L^2(B_R))} &\lec_{a,M,u_0 }  t^{\frac1p}  \ee^{CM^4t} R^{-a} 
 \end{split}
\eqne
for all $t\in [0,T]$, $a\in [0,3/2)$ and $p\in (1,\infty )$.
\end{thm}
In the above theorem and below we use subscripts to articulate any dependencies of implicit constants. For example, the symbol ``$\lec_{a,M,u_0 }$''  denotes ``$\leq C(a,M,u_0)$'' for some implicit constant $C(a,M,u_0)>0$ dependent on $a,M,u_0$ only. 

In fact, in the particular case of \eqref{how_large_R} the constant can be made more precise,
\[
C(a,M,u_0) \coloneqq C(a) \left( \frac{\cC^2 }{M^3} \right)^{-\frac1a}
\]
where 
\eqnb\label{def_of_Cl}
\cC \coloneqq \ee^{C MT} (4R_0^{2a+8} \| u_0 \|^2_{H^{5}} + M).
\eqne
Similarly the implicit constants in \eqref{main_ests} can be taken of the form $C(a,M,u_0) = C(a) M \cC^2$. These quantified constants are clear from the proof below.

We note that the assumptions of the above theorem hold if $u \in L^\infty ((0,T);L^\infty )$ or $\nabla u \in L^\infty ((0,T); H^1)$ (as shown by Leray \cite{leray_34}, see also \cite[Section~6.3]{op}). In fact in that case $M \leq C (\| u \|_{L^\infty ((0,T);L^\infty )}, \| u_0 \|_{H^5 },\| u_0 \|_{W^{5,\infty } })$, which can be shown by an iteration (with respect to the order of the derivatives) of Gronwall inequalities (see \cite{constantin_foias} or \cite[Theorem~7.1]{NSE_book} for details). There are a number of well-known sufficient conditions that guarantee that a given Leray weak solution is in fact strong \cite{chae_wolf,cf_vorticity,ko_d3u,ladyzhenskaya, np1,prodi,NSE_book,serrin,skalak,wwz}. One of them is the Ladyzhenskaya-Prodi-Serrin \cite{ladyzhenskaya,prodi,serrin} condition, $u \in L^p ((0,T);L^q (\RR^3 ))$, where $p\in [2,\infty )$, $q\in (3,\infty ]$ are such that
\[
\frac2p+\frac3q =1.
\]
Then in fact
\[
\| \nabla u \|_{L^\infty ((0,T);L^2)}^2 \leq \|\nabla u_0 \|_2^2 \exp \left( \| u \|_{L^p ((0,T);L^q (\RR^3))}^p\right),
\]
see \cite[p.~173]{NSE_book}, for example.

The above theorem is also valid with $B_R$ replaced by any bounded and smooth domain of the form $R\Omega $, where $B_1\subset \Omega $. 

Since Theorem~\ref{thm_main} can be proved using a straightforward procedure, we present it now, including some quantitative bounds that will be verified in detail in Sections~\ref{sec_spatial_decay}--\ref{sec_NSE_with_forcing} below. We use the notation $\| \cdot \|_p \equiv \| \cdot \|_{L^p (\RR^3)}$ for brevity.

\begin{proof}[Proof of Theorem~\ref{thm_main}]
 Let $\phi \in C_0^\infty (\RR^3 ; [0,1])$ be such that $\phi =1$ on $B_{R-1}$, and $\phi=0$ outside $B_{R}$. Then $u\phi$ is a solution to the problem
\[
\begin{split}
\p_t (u\phi ) - \Delta (u\phi ) + \na (\phi \pi ) + ((u\phi ) \cdot \na ) (u\phi ) & = F_1    \\
\left. (u\phi ) \right|_{\partial B_R} &= 0,\\
\div \, (u\phi) &= \na \phi \cdot u,\\
(u\phi) (0) &= u_0
\end{split}
\]
in $\RR^3 \times (0,T)$,  where
\[
F_1 \coloneqq - \phi(1-\phi ) (u\cdot \na )u + (\phi u \cdot \na \phi ) u +  \pi \na \phi  - u \Delta \phi  - 2 \na u \cdot \na \phi.
\]
In Section~\ref{sec_spatial_decay} we study the spatial decay of strong solutions of \eqref{NSE_intro} to deduce that at each time $t\in [0,T]$
\eqnb\label{to_show_decay1}
\|\nabla u  - \nabla (u\phi ) \|_2 \lec_a \cC R^{-a} \qquad \text{ and } \qquad \| \nabla ((1-\phi ) \pi ) \|_2\lec_a \cC^2 R^{-a}
\eqne
and
\eqnb\label{decay_of_F1}
\| F_1  \|_2 \lec_a \cC^2 R^{-a },
\eqne
where $a \in [0,3/2)$.

Since $\div (u\phi ) \ne 0$, we can use a Bogovski\u{\i}-type correction $u_c$ (see Section~\ref{sec_bogovskii}) such that 
\eqnb\label{bound_on_nabla_uc}
\| \nabla u_c \|_2 + \| u_c \|_\infty \lec_a \cC R^{-a},
\eqne
\eqnb\label{decay_of_F2}
\hspace{1.5cm}\| F_2  \|_2 \lec_a \cC^2 R^{-a }
\eqne
for all $t\in [0,T]$ and $a \in [0,3/2)$, where
\[
F_2 \coloneqq (u_c \cdot \na ) (u\phi ) + (u\phi \cdot \na )u_c + (u_c \cdot \na ) u_c - \Delta u_c +\p_t u_c,
\]
and such that the ``corrected velocity field''
\eqnb\label{def_of_r}
r\coloneqq u\phi + u_c
\eqne
is divergence free and vanishes on $\p B_R$.

 Letting
\[
F\coloneqq F_1+F_2,
\]
we deduce from \eqref{decay_of_F1}, \eqref{decay_of_F2} (in Section~\ref{sec_NSE_with_forcing}) that for $R$ satisfying  \eqref{how_large_R} the problem
\eqnb\label{NSE_general_form_intro}
\begin{split}
\p_t v - \Delta v + (v\cdot \nabla ) v + \nabla \overline{ \pi } &=F- (v\cdot \nabla )r - (r\cdot \nabla ) v \qquad \text{ in } B_R \times (0,T), \\
\div v &=0 \qquad \text{ in } B_R \times [0,T],\\
\left. v(t) \right|_{\p B_R} &= 0 \qquad \text{ for } t\in [0,T]
\end{split}
\eqne
has a unique strong solution $v$ on $[0,T]$ with $v(0)=0$ (see Definition~\ref{def_strong_sol} and Lemma~\ref{lem_solving_nse_with_forcing}), such that
\eqnb\label{smallness_of_nabla_v}
\| \nabla v(t) \|_{L^2 (B_R)} \lec_a \frac{\cC^{2} R^{-a}  }{M^{2}}  \ee^{CM^{4} t} \quad \text{ and }\quad \| \nabla \overline{\pi } \|_{L^p((0,t);L^2(B_R))} \lec_{a,p} t^\frac1p M \cC^2 R^{-a} \ee^{CM_2^4 t } ,
\eqne
for every $t\in [0,T]$, $p\in (1,\infty )$, $a\in [0,3/2)$, where $C>1$ is an absolute constant. It follows that then 
\[
w\coloneqq r + v,\qquad \tilde{\pi} \coloneqq \pi \phi + \overline{\pi}
\]
satisfies the claim of Theorem~\ref{thm_main} (i.e., satisfies \eqref{NSE_general_form_intro} with vanishing right-hand side and initial condition $w(0)=u_0$). The estimates \eqref{main_ests} follow directly from \eqref{to_show_decay1} and \eqref{smallness_of_nabla_v}. \end{proof}

The most difficult part of the above procedure are the decay estimates \eqref{to_show_decay1},\eqref{decay_of_F1}, for which we employ the machinery developed by Kukavica and Torres \cite{kt_06,kt_07}. Inspired by \cite{kt_07} we consider the vorticity equation (on the whole space $\RR^3$), and we observe that, given spatial derivatives of $u$ are bounded, all spatial derivatives of vorticity $\omega $ are well-localized. This allows us to control $\| |x|^a D^l \omega \|_2$ for all $l$ and $a$ (see \eqref{bound_on_Gl}). In order to obtain control in other $L^p$ spaces of $|x|^a D^lu$ we need to apply a Fourier method (see Lemma~\ref{lem_fourier} and \eqref{lp_decay_of_derivatives}). This also results in the restriction $p\in [2,\infty)$ and $a\in [0,3/p'+l)$, where $p' $ denotes the dual exponent of $p$, and $l\geq 1$ denotes the order of the spatial derivatives of $u$.  Using the Caffarelli-Kohn-Nirenberg inequality (see \eqref{ckn_ineq} below), we can also handle the case $l=0$ (see \eqref{lp_decay_of_derivatives}). Using weighted pressure inequalities \eqref{pressure_decay1_from_igor} and \eqref{pressure_decay_from_igor}, we can then obtain \eqref{decay_of_F1}, see Section~\ref{sec_spatial_decay}. The reason for the restriction $a\in [0,3/2)$ comes from the lowest order terms $\pi \nabla \phi $ and $u\Delta \phi$ appearing in $F_1$, as then taking $p\coloneqq 2$ gives $a\in [0,3/2)$.

The main feature of the Bogovski\u{\i} correction $u_c$ is that it is supported on the set $B_R \setminus B_{R-1}$, which is not star-shaped. Such extensions of the Bogovski\u{\i} lemma are well-known (see~\cite{borchers_sohr} or \cite[Section~III.3]{galdi_book}, for example). In our case the divergence structure of $\mathrm{div}\, (u\phi ) = \nabla \phi \cdot u$ make such extension easier by using a partition of identity to decompose $\phi$ into a number of cutoff functions supported on star-shaped domains. We discuss this and show the resulting estimates \eqref{bound_on_nabla_uc} and \eqref{decay_of_F2} in Section~\ref{sec_bogovskii}.

Finally, the well-posedness of the system \eqref{NSE_general_form_intro} can be verified using classical arguments. We discuss it in Section~\ref{sec_NSE_with_forcing} in order to expose the required smallness of $F$ and to verify  \eqref{smallness_of_nabla_v}. 

We note that the required order $5$ of derivatives that must be under control in the assumptions in Theorem~\ref{thm_main} come from the fact that the highest order derivative of $u$ whose spatial decay must be under control is $2$ (see term ``$\Delta u_c$'' in $F_2$). Due to our  Fourier method in Lemma~\ref{lem_fourier}, this translates into decay estimate on $D^4 \omega$ (see \eqref{lp_decay_of_derivatives}), which in turn requires boundedness of $\| u \|_{H^5}+\| u \|_{W^{5,\infty }}$ (see \eqref{bound_on_Gl}).

We also note that the proof of Theorem~\ref{thm_main} can also be performed with the cutoff function $\phi$ replaced by a cutoff with a different transition rate, such as $\phi_1 \in C_0^\infty ( B_{R_2} ; [0,1])$ with $\phi_1=1$ on $B_{R_1}$, where $R_2>0$, $R_1\in (0, R_2)$. In such case some additional factors of $(R_2-R_1)^{-1}$ can be obtained from the terms involving derivatives of $\phi_1$. Thus the main estimate \eqref{how_large_R} would become weaker as $|R_2-R_1|\to 0$. On the other hand it would not improve as $|R_2 - R_1|\to \infty$ in the sense that we would still need $R_1$ bounded below as in \eqref{how_large_R}. For example, we use the inequality $|x|>R_1$ whenever $(1-\phi_1)\ne 0$ (e.g. in \eqref{ex_of_optimality} below), which would make the decay estimates \eqref{to_show_decay1}-\eqref{decay_of_F2} at least of order $R_1^{-a}$. \\

It will become clear from the proof (see \eqref{bound_on_Gl}) that one can replace \eqref{def_of_Cl} with 
\eqnb\label{def_of_cC}
\cC \coloneqq \ee^{CMT} ( A+ M),
\eqne
where $C>1$ is a universal constant, and $A>1$ is such that 
\eqnb\label{assump_on_A_B} 
\| |x|^{a+4} D^l \omega_0  \|^2_2 \leq A
\eqne
for $l=0,\ldots ,4$, where $\omega_0 \coloneqq \mathrm{curl}\, u_0$ denotes the initial vorticity (see \eqref{bound_on_Gl} below).

Thus the requirements $u_0\in C_0^\infty (B_{R_0})$ and $R\geq R_0 +1$ are not necessary in Theorem~\ref{thm_main}, and we obtain the following.
\begin{cor}[Non-compactly supported $u_0$]\label{cor_main}
Let $a\in [0,3/2)$. Suppose that $u_0 \in H^5 (\RR^3 ) \cap W^{5,\infty } (\RR^3)$ is such that \eqref{assump_on_A_B} holds for some $A>1$. Then the claim of Theorem~\ref{thm_main} remains valid for 
\[ R\geq C(a) \left( \frac{\cC^2} {M^3} \ee^{CM^4 T } \right)^{\frac1a},
\] 
where $\cC$ is from \eqref{def_of_cC}, with the initial condition on $w(0)$ replaced by $w(0)=u\phi + u_c$ (see \eqref{def_of_r} above).
\end{cor}

\section{Proof of the main result}
As sketched above, Theorem~\ref{thm_main} follows from \eqref{to_show_decay1}--\eqref{smallness_of_nabla_v}.  We first introduce some concepts and inequalities.
\subsection{Preliminaries}\label{sec_prelims}
We use standard conventions regarding the Lebesgue spaces $L^p (\Omega )$, Sobolev spaces $W^{k,p}(\Omega )$, $H^k(\Omega )$, $H^1_0(\Omega )$, and we denote by $C_0^\infty (\Omega )$ the space of smooth functions with compact support in $\Omega $. We write $\| \cdot \|_p \equiv \| \cdot \|_{L^p}$. In the case $\Omega = \RR^3$ we omit the domain to simply write $L^p \equiv L^p(\RR^3)$ and $H^k \equiv H^k_0 \equiv H^k (\RR^3)$. 
We denote by $V$ the closure of the set of divergence-free functions $\phi\in C_0^\infty (\Omega )$ in the $H^1$ norm. We denote by $\Omega^c \coloneqq \RR^3 \setminus \Omega $ the complement of a set $\Omega $. We denote by $C\geq 1$ any universal constant that may change value from line to line.

We denote by 
\[
\widehat{f} (\xi  ) \coloneqq \int_{\RR^3} f(x) \ee^{-2\pi i x\cdot \xi } \d x
\]
the Fourier transform of $f$, and by $\cR_j f $, $\widehat{\cR_j f }(\xi ) \coloneqq \cR_j (\xi )\widehat{f}(\xi ) \equiv \frac{\xi_j}{|\xi |} \widehat{f}(\xi )$, the Riesz transform with respect to the $j$-th variable, for $j=1,2,3$. We let $\eta\in C_0^\infty (B_2; [0,1])$ be such that $\eta =1$ on $B_1$, and let $\widetilde{\eta } \coloneqq 1- \eta $. We recall that 
\eqnb\label{riesz_of_outer_cutoff}
\| \Lambda^a ( \cR (\xi ) \widetilde{\eta } (\xi ) ) \|_\infty <\infty
\eqne
for every $a\geq 0$, see \cite[Lemma~2.6]{kt_07} for a proof.

Given $T>0$ and a smooth and bounded domain $\Omega \subset \RR^3$, we consider the Navier-Stokes initial boundary value problem,
\eqnb\label{NSE_general_form}
\begin{split}
\p_t v - \Delta v + (v\cdot \nabla ) v + \nabla \overline{\pi } &=F- (v\cdot \nabla )r - (r\cdot \nabla ) v \qquad \text{ in } \Omega \times (0,T),\\
\div v &=0,\\
\left. v \right|_{\p \Omega} &= 0,\\
v(0) &=v_0,
\end{split}
\eqne
where $v_0 \in C_0^\infty (\Omega )$, $F\in L^\infty ((0,T);L^2)$ and $r\in L^2([0,T]; H^1 )$.

\begin{defn}[Strong solution to \eqref{NSE_general_form}]\label{def_strong_sol}
We say that $v$ is a strong solution to \eqref{NSE_general_form} if $v\in L^\infty ((0,T);V)\cap L^2 ((0,T);H^2 (\Omega ))$ and
\eqnb\label{distr_form_of_eqs}
\int_0^t \int_\Omega \left( -v\cdot \p_t \varphi + \nabla v : \nabla \varphi + \left( (v\cdot \nabla ) v + (v\cdot \nabla ) r + (r\cdot \nabla )v -F\right) \cdot \varphi  \right) = \int_\Omega v_0 \cdot \varphi (0) - \int_\Omega v(t)\cdot \varphi (t)
\eqne
for all divergence-free $\varphi \in C_0^\infty ([0,\infty )\times \Omega ) $ and almost all $s\in (0,T)$.
\end{defn}
We note that for every strong solution $v$ to \eqref{NSE_general_form}, there exists a unique (up to a function of time) pressure function $\overline{\pi }$ (see \cite{svw_86} or \cite[Chapter~5]{NSE_book}). In the case of $\Omega = \RR^3$ 
\eqnb\label{pressure_formula_in_R3}
\overline{\pi }  = \sum_{i,j=1}^3 \cR_i \cR_j (u_j u_i),
\eqne
see \cite[(6.47)]{op} or \cite[Section~5.1]{NSE_book} for details. Moreover,
\eqnb\label{svw_ineq}
\| \nabla \overline{\pi } \|_{L^p ((0,t); L^q (\Omega ) )}  \lec_{p,q} \| F - (v\cdot \nabla )r - (r\cdot \nabla )v  \|_{L^p ((0,t); L^q (\Omega ) )} 
\eqne
for $p,q\in (1,\infty )$, see \cite[Theorem~2.12]{svw_86}. We note that for $\Omega = B_R$ the implicit constant in \eqref{svw_ineq} does not depend on $R$, which can be verified by a scaling argument.

Given $R>0$ we denote by $\P: L^2 (B_R)\to L^2 (B_R)$ the standard Leray projection, i.e. 
\[
\P v \coloneqq v - \nabla \phi,
\]
where $\phi \in H^1_0 (B_R)$ is the unique weak solution of the Poisson equation $\Delta \phi = \div u$ with the homogeneous boundary condition $\left. \phi \right|_{\partial B_R } =0$. Then $\| \P v \|_{L^2 (B_R)} \leq \| v \|_{L^2(B_R)}$. It follows from the uniqueness of weak solutions of the Poisson equation that, if $v_\lambda (x) \coloneqq v (\lambda x)$, then
\eqnb\label{dilation}
(\P v )_\lambda = \P v_\lambda\quad \text{ and }\quad  \| (\P v )_\lambda \|_{L^2 (B_{R/\lambda })} = \lambda^{-3/2}\| \P v \|_{L^2 (B_{R })}.
\eqne
Given $R>0$ we also set
\eqnb\label{stokes_op_def}
D(A) \coloneqq H^1_0 (B_R) \cap H^2 (B_R ),\quad \text{ and } \quad Au\coloneqq \P \Delta u\qquad \text{ for }u\in D(A).
\eqne
Now recall the homogeneous Agmon's inequality
\eqnb\label{agmon}
\| u \|_{L^\infty (B_R)} \leq C \| \nabla u \|_{L^2 (B_R)}^\frac12 \| A u \|_{L^2 (B_R)}^\frac12 
\eqne
for $u\in D(A)$, where $C>0$ is a constant that does not depend on $R$. Indeed, in the case $R=1$ the inequality follows by $\| u \|_{L^\infty } \leq \tilde{C} \| u \|_{H^1}^\frac12 \| u \|_{H^2}^\frac12$ (see, for example, Theorem 1.20 in \cite{NSE_book}) by applying the Poincar\`e inequality to replace $\| u \|_{H^1}$ by $\|\nabla u \|_{L^2}$ and by applying a Stokes estimate to replace the last norm by $\| A u \|$ (see, for example, Proposition~2.2 in Temam \cite{temam_book}). The case of $R\ne 1$ follows by rescaling and observing \eqref{dilation}.
Another application of the Stokes estimate and an observation of the scaling gives that
\eqnb\label{homo_stokes_op}
\|D^2  u \|_{L^2 (B_R) } \sim \| Au \|_{L^2 (B_R)},
\eqne
for all $u\in H^1_0 (B_R)\cap H^2 (B_R)$, where the symbol $\sim$ means ``$\lec$ and $\gtrsim$'', and the implicit constants are independent of $R>0$. In a similar way we obtain
\eqnb\label{interp1}
\| u \|_{L^4 (B_R)} \leq C \|  u \|_{L^2 (B_R)}^\frac14 \| \nabla  u \|_{L^2 (B_R)}^\frac34 
\eqne
for all $u\in H^1_0 (B_R)$, where $C>0$ does not depend on $R$.

We recall the weighted inequality for singular integrals (see \cite{stein_singular_int} or  \cite[(2.15)]{kt_07}),
\eqnb\label{weighted_singular_integral}
\| |x|^a \na u \|_2 \lec_a \| |x|^a \omega \|_2
\eqne
for $a\in [0,3/2)$, as well as the Caffarelli-Kohn-Nirenberg inequality (see \cite{ckn_84}),
\eqnb\label{ckn_ineq}
\| |x|^{a-1} u \|_p \lec_{a,p} \| |x|^a \na u \|_p,
\eqne
where $a\in [0,\infty)$. We will also use the inequality of Gruji\'{c} and Kukavica \cite{gk_03}
\eqnb\label{kukavica_grujic}
\| f \|_1 \leq C \| f \|_2^{\frac12} \| |x|^3 f \|_2^{\frac12},
\eqne
as well as the inequality due to Chae \cite{chae_2002},
\eqnb\label{chae_ineq}
\| \Lambda^a (fg) \|_p \lec_p \| f \|_{p_1} \| \Lambda^a g \|_{p_2} + \|\Lambda^a f \|_{q_1} \|  g \|_{q_2}  ,
\eqne
where $a>0$, $p\in (1,\infty )$ and $p_1,p_2,q_1,q_2\in [1,\infty ]$ are such that $1/p_1+1/p_2 =1/q_1+1/q_2=1/p$. \\ Moreover, for $p\in (1,\infty )$,
\eqnb\label{pressure_decay1_from_igor}
\| |x|^a \pi\|_p \lec_{p,a} \|  |x|^a |u|^2 \|_p
\eqne
for $a\in [0, n/p')$, and 
\eqnb\label{pressure_decay_from_igor}
\|   |x|^a \na \pi \|_p \lec_{p,a} \| |x|^a  | u | \,| \na u | \|_p + \|  |x|^{a-1} |u|^2\|_p
\eqne
for $a\in [0, n/p'+1)$, where the last term can be omitted if $a<n/p'$, see  \cite[Lemma~4.1 and Lemma~4.2]{k_01} for a proof. 

 \subsection{Spatial decay of strong solutions in $\RR^3$}\label{sec_spatial_decay}
In this section we are concerned with the decay properties of strong solutions to \eqref{NSE_intro} on the whole space $\RR^3$, and we prove \eqref{to_show_decay1} and \eqref{decay_of_F1}.

To this end we note that the vorticity $\omega \coloneqq \mathrm{curl} \, u$
satisfies
\[
\p_t \omega - \Delta \omega + (u \cdot \nabla ) \omega - (\omega \cdot \nabla ) u =0.
\]
Since $D^l u$ is bounded in time in any $L^p$ we see that, considering $u$ as given, the vorticity equation is local, which enables us to control any spatial decay of any spatial derivative of $\omega$. To be more precise, we let \[
G_l(t)\coloneqq \| |x|^a D^l \omega \|_2^2,
\]
and we observe that, for each multiindex $\alpha$ with $|\alpha | =l$,
\[
\begin{split}
\frac{\d }{\d t} \left( \int |x|^{2a} |D^\alpha \omega |^2 \right) &= 2 \int |x|^{2a} D^\alpha \omega_j \p_t D^\alpha \omega_j \\
&= 2\int |x|^{2a} D^\alpha \omega_j \Delta D^\alpha \omega_j - 2 \int |x|^{2a} D^\alpha \omega_j D^\alpha \p_k ( u_k \omega_j -\omega_k u_j) \\
&\leq -  \int |x|^{2a} |\na D^\alpha \omega |^2 -4a \int |x|^{2a-2} x_k D^\alpha \omega_j \p_k D^\alpha \omega_j +  \|  u \|_{W^{l+1,\infty}} G_l(t)^\frac12 G_{l+1}(t)^\frac12 \\
&\leq -  \frac12 \int |x|^{2a} |\na D^\alpha \omega |^2 + C\int |x|^{2a-2} | D^\alpha \omega |^2 + \|  u \|_{W^{l+1,\infty}} G_l(t)^\frac12 G_{l+1}(t)^\frac12 \\
&\leq -  \frac12 \int |x|^{2a} |\na D^\alpha \omega |^2 + C\| D^\alpha \omega \|_2^{\frac2a} G_l(t)^\frac{a-1}a  + \|  u \|_{W^{l+1,\infty}} G_l(t)^\frac12 G_{l+1}(t)^\frac12 ,
\end{split}
\]
where, in the third line, we integrated the first term by parts, and, in the fourth line, we noted that $a\leq 10$ and applied the Young inequality $cb\leq \varepsilon c^2 + C_\varepsilon b^2 $ to absorb a part of the second term by the first term. We also applied H\"older's inequality in the last line. Thus, summing in $|\alpha | =l$, for $l=1,\ldots ,4$,
\[
G_l'(t) \lec - G_{l+1}(t) + \| \omega \|_{W^{l,2}}^{\frac2a} G_l (t)^{\frac{a-1}a} + \|  u \|_{W^{l+1,\infty}} G_l(t)^\frac12 G_{l+1}(t)^\frac12  \lec M +  (1 + M ) G_l(t),
\]
where we have also applied the Young inequality. By the Gronwall inequality we obtain that
\eqnb\label{bound_on_Gl}
G_l (t) \lec  \ee^{CMt} (G_l (0 )+ M) \leq \cC
\eqne
for $t\in [0,T]$, $l=0,\ldots 4$, where we have recalled \eqref{def_of_Cl} for the definition of $\cC_l$ and noted that $\| u (t) \|_{H^{l}},\| u (t) \|_{W^{l,\infty}}\leq M_{l}$ for all $t\in [0,T]$ (recall Theorem~\ref{thm_main}). Moreover we used the assumption that $M\geq 1$.

In order to translate this estimate into decay of the velocity field $u$ in $L^p$ for $p\geq 2$, we need the following lemma, which is concerned with homogeneous Fourier multipliers $\M$ of the form $\M = \cR_i \cR_j \p^\beta$, where $\cR_j$ stands for the Riesz transform with respect to the $j$-th variable (recall Section~\ref{sec_prelims}) and $\beta $ is a multiindex. In other words 
\eqnb\label{form_of_m} \widehat{\M f }(\xi ) = m(\xi ) \widehat{f} (\xi)\equiv \frac{\xi_i \xi_j }{|\xi |^2} \xi^\beta \widehat{f}(\xi )\equiv \cR (\xi ) \xi^\beta \widehat{f}(\xi ),
\eqne
where $\beta \in \N^l $ for some $l\in \N$.

\begin{lem}\label{lem_fourier} Suppose that $\int g =0$ and that $\| |x|^a \p^\alpha g \|_2 \lec_a M_{k}$ for every $a\geq 0 $ and every multiindex $\alpha$ with $|\alpha |\leq k$. Then
\[
\| |x|^a \M g \|_p \lec_{a,p} M_{l+3}
\]
for every $p\in [2, \infty )$, $a\in [0,3/p'+l+1)$, where $\M $ is a multiplier of the form \eqref{form_of_m} with $|\beta | =l$.
\end{lem}
\begin{proof}
The proof is inspired by Lemma~2.8 in \cite{kt_07}. Recall (from Section~\ref{sec_prelims}) that $\eta \in C_0^\infty (\RR^3 ; [0,1])$ is such that $\eta =1$ on $B(1)$ and $\eta =0$ outside $B(2)$, and $\tilde{\eta } \coloneqq 1-\eta $.\\

By the Hausdorff-Young inequality (see \cite[Theorem~1]{beckner}),
\eqnb\label{loc_and_nonloc}\begin{split}
\| |x|^a \M g \|_p &\leq \| \Lambda^a (m(\xi) \widehat{g} (\xi ) )\|_{p'}\leq \| \Lambda^a (m(\xi) \eta(\xi )\widehat{g} (\xi ) ) \|_{p'}+\| \Lambda^a (m(\xi) \tilde{\eta}(\xi )\widehat{g} (\xi ) ) \|_{p'}.
\end{split}
\eqne
Using \eqref{chae_ineq} and recalling the form of $m$ \eqref{form_of_m} we can estimate the second term on the right-hand side by a constant multiple of
\[
\| \xi^\beta \widehat{g}(\xi )  \|_{p'} \| \Lambda^a (\cR (\xi ) \tilde{\eta } (\xi )) \|_\infty + \| \cR(\xi ) \tilde{\eta } (\xi ) \|_\infty  \| \Lambda^a ( \xi^\beta  \widehat{g}(\xi ) ) \|_{p'} \lec_m \| \xi^\beta \widehat{g}(\xi )  \|_{p'}+\| \Lambda^a ( \xi^\beta  \widehat{g}(\xi ) ) \|_{p'},
\]
where we used \eqref{riesz_of_outer_cutoff} in the last step. In order to estimate the resulting terms, we first replace $\Lambda^a$ by a classical derivative $\p^\gamma$, for $\gamma \in \N$, and observe that, since $p'\in (1,2]$, we can use Lebesgue interpolation and the inequality \eqref{kukavica_grujic} of Gruji\'{c} and Kukavica to get
\[
\| \p^\gamma_\xi (\xi^\beta \widehat{g}(\xi ) \|_{p'} \lec \sum_{\beta+\gamma'=\beta'+\gamma} \| \xi^{\beta'} \p_\xi^{\gamma'} \widehat{g} (\xi ) \|_{p'} \lec \sum_{\beta+\gamma'=\beta'+\gamma} \| \xi^{\beta'} \p_\xi^{\gamma'} \widehat{g} (\xi ) \|_{2}^{\frac12+\frac1p} \| |\xi|^{3+|\beta'|} \p_\xi^{\gamma'} \widehat{g} (\xi ) \|_{2}^{\frac12-\frac1p}  \lec_a M_{l+3},
\]
where in the last step we used the Plancherel identity and the assumption to note that
\[
\| |\xi |^b \p_\xi^{\gamma'} \widehat{g} (\xi )\|_2 \leq \sup_{|\kappa |=b} \| \p_x^\kappa (x^{\gamma'} g) \|_2 \lec_{\gamma'} M_{b}
\]
for every $b\geq 0$, and every multiindex $\gamma'$. Thus also 
\eqnb\label{interpolation_1}\| \Lambda^a ( \xi^\beta  \widehat{g}(\xi ) ) \|_{p'}\lec_{p,a}   M_{l+3},
\eqne
by interpolation. 

It remains to estimate the first term on the right-hand side of \eqref{loc_and_nonloc},
\[
\| \Lambda^a (m(\xi) \eta(\xi )\widehat{g} (\xi )) \|_{p'}.
\]
 To this end we note that by assumption $\widehat{g}(0)=0$, and so the Fundamental Theorem of Calculus gives
 \[
 \widehat{g}(\xi ) = \xi \cdot \int_0^1  \na_\xi \widehat{g} (s\xi )(1-s) \d s.
 \]
This and \eqref{chae_ineq} gives that
\[\begin{split}
\| \Lambda^a (m(\xi) \eta(\xi )\widehat{g} (\xi ) \|_{p'} &\leq  \| m(\xi ) \xi \eta ( \xi ) \|_{p'} \left\| \Lambda^a \int_0^1 (1-s) \na \widehat{g} (s\xi ) \d s \right\|_{\infty} \\
&\hspace{1cm}+ \| \Lambda^a (m(\xi )\xi  \eta (\xi ))\|_{p'} \left\| \int_0^1 (1-s) \na \widehat{g} (s\xi ) \d s \right\|_\infty\\
&\lec_a  \|  \Lambda^a \na \widehat{g}  \|_{\infty} + \| \na \widehat{g} \|_\infty \\
&\lec  \| |x|^{a+1} g \|_1 + \| |x| g \|_1 \lec_a M_{0},
\end{split}
\] 
 where we used the fact that $a\in [0,3/p'+l+1)$ to deduce that $\| \Lambda^a (m(\xi ) \xi_i \eta (\xi ) \|_{p'}\lec_a 1$ in the second inequality as well as the Gruji\'{c}-Kukavica inequality \eqref{kukavica_grujic} in the last step.
\end{proof}
Noting that $\mathrm{curl}\, \omega =\mathrm{curl}\,\mathrm{curl}\,u= \nabla (\div u ) - \Delta u= -\Delta u$, we obtain  
\[
\p_m u_i = \p_m (-\Delta )^{-1} (\mathrm{curl}\, \omega )_i = \p_m (-\Delta )^{-1} \sum_{j,k=1}^3 \epsilon_{ijk} \p_j \omega_k = \sum_{j,k=1}^3 \epsilon_{ijk} \cR_m \cR_j \omega_k,
\]
where $\epsilon_{ijk}$ denote the Levi-Civita tensor, that is $\epsilon_{ijk}=1$ if $ijk$ is an even permutation of $123$, $-1$ if odd, and $0$ otherwise. Thus we can apply the above lemma with $g\coloneqq \omega_k $, $\mathcal{M} \coloneqq \cR_m cR_j \p^\beta$, where $\beta \in \mathbb{N}^{l-1}$, $j,k,m\in \{ 1,2,3\}$, and obtain 
\eqnb\label{lp_decay_of_derivatives}
 \| |x|^a D^l u \|_p \lec_{p,a} \cC
\eqne
for $t\in [0,T]$ (which we omit in our notation), $p\in [2,\infty )$, $l=1,2$ and $a\in [0, 3/p' +l)$. The case $l=0$ follows by the Caffarelli-Kohn-Nirenberg inequality \eqref{ckn_ineq}.\\
In particular we obtain the first claim in \eqref{to_show_decay1} as
\eqnb\label{ex_of_optimality}
\| \nabla u - \nabla (u\phi ) \|_2 \leq \| (1-\phi ) \nabla u \|_2 + \| \nabla \phi u \|_2  \lec R^{-a} \left( \| |x|^a \nabla u \|_2 +  \| |x|^a u \|  \right) \lec_a \cC R^{-a}
\eqne
for $a\in [0,3/2)$. Another consequence of \eqref{lp_decay_of_derivatives} is that 
\[
\| |x|^a (u\cdot \nabla ) u \|_2 \leq \| |x|^{a_1} u \|_{4} \| |x|^{a_1} \nabla u \|_4 \lec_a \cC^2,  
\]
for $a\in [0,11/2)$, where $a_1\in [0,9/4)$, $a_2\in [0,13/4)$ are such that $a=a_1+a_2$.
Moreover,  \eqref{pressure_decay1_from_igor} gives
\eqnb\label{decay_of_p}
\| |x|^a  \pi \|_p \lec_{p,a} \| | u |^2 |x|^a \|_p  \lec_{p,a}  \cC^2
\eqne
for $p>1$, $a\in [0,3/p')$, while  \eqref{pressure_decay_from_igor} implies
\[
\| |x|^a \nabla \pi \|_p \lec_{p,a} \| | u | \,| \na u |\, |x|^a \|_p + \| |u|^2 |x|^{a-1} \|_p \leq \| u \|_\infty \left( \|  \na u  |x|^a \|_p + \|u |x|^{a-1} \|_p \right) \lec_{p,a}  \cC^2
\]
for $p>1$, $a\in [0,n/p'+1)$. Thus in particular
\[
\| \nabla ((1-\phi ) \pi ) \|_2 \lec \| \pi \|_{L^2(B_{R-1}^c)} + \| \nabla \pi \|_{L^2 (B_{R-1}^c)} \lec_a \cC^2 R^{-a}
\]
for $a\in [0,3/2)$, which gives the second claim in \eqref{to_show_decay1}.

The above estimates also imply \eqref{decay_of_F1}, as
\[\begin{split}
\| F_1 \|_2 &\leq \| |(u\cdot \nabla ) u |+| u |^2  +| p | +| u| + |\na u | \|_{L^2 (B_{R+1}\setminus B_R)}\\
& \leq R^{-a}  \||x|^a ( |(u\cdot \nabla ) u |+| u |^2  +| \pi | +| u| + |\na u | )\|_{2} \lec_a \cC^2 R^{-a}
\end{split}
\]
for every $a<3/2$.

Moreover, we can also use the Navier-Stokes equations \eqref{NSE_intro} to estimate the decay of $u_t$ as
\eqnb\label{decay_of_ut}
\| |x|^a \p_t u \|_2 \leq \| |x|^a \Delta  u \|_2 + \| |x|^a (u\cdot \nabla)u  \|_2 + \| |x|^a \nabla \pi \|_2 \lec_a \cC^2
\eqne
for $a\in [0,5/2)$. 

\subsection{The Bogovski\u{\i}-type correction}\label{sec_bogovskii}
In this section we consider the correction $u_c$ that makes $r \coloneqq u \phi + u_c $ divergence free, recall \eqref{bound_on_nabla_uc}, \eqref{decay_of_F2}. We first recall the Bogovski\u{\i} lemma.

\begin{lemma}[Bogovski\u{\i} lemma]\label{lem_bogovski}
Let $\Omega \subset \RR^3$ be a star-shaped domain with respect to $B_1$ (namely that the line segment $[x,y]$ joining any $x\in \Omega$ with any $y\in B_1$ is contained in $\Omega $). Then given $f\in C_0^\infty (\Omega )$ with $\int f =0$ there exists $v\in C_0^\infty ( \Omega ;\RR^3)$ such that $\mathrm{div}\, v =f$ and 
\[
\| v \|_{W^{k,p}(\Omega )} \lec_{k,p}  \| f \|_{W^{k-1,p} (\Omega )}.
\]
\end{lemma}
The Bogovski\u{\i} lemma is a well-known result (see \cite{bogovskii_79,bogovskii_80} or \cite[Lemma~III.3.1]{galdi_book}, for example). In fact, letting  $h\in C_0^\infty (B_1)$ be such that $\int h=1$, the vector field
\eqnb\label{bogovskii_formula}
v(x) \coloneqq \int_\Omega f(y) \left( \frac{x-y}{|x-y|^3} \int_{|x-y|}^\infty h\left( y + z \frac{x-y}{|x-y|} \right) z^2 \d z \right) \d y 
\eqne
satisfies the claim of the lemma.

We note that our domain, $B_{R}\setminus B_{R-1}$, is not star-shaped, and so we need to decompose the domain as well as $f=\nabla \phi \cdot u$ into a number of pieces that would allow us to construct the correction $u_c$. Some decompositions of this form can be found in \cite[Lemma~III.3.2 and Lemma~III.3.4]{galdi_book}, where one of the main difficulties is to guarantee that each of the pieces still have compact support as well as vanishing integral. In our case, this issue simplifies, as $f=\div (u \phi)$, and so this divergence structure allows us to apply a partition of unity inside ``div''. To be more precise, we let $G_1,\ldots , G_L\subset B_R \setminus B_{R-1}$ be open balls, and $\psi_l \in C_0^\infty (G_l;[0,1])$ ($l=1,\ldots , L$) be such that $\| \psi_l \|_{W^{2,\infty }} \leq C$ for some universal constant $C>0$ and $\psi_1+\ldots \psi_L=1$ on $\mathrm{supp}\,\phi$. We can assume that each $G_l$ intersects at most $10$ other $G_l$'s. 

Note that $\div (\psi_l \phi u)$ is compactly supported and has vanishing mean for each $l=1,\ldots , L$ and so we can use Lemma~\ref{lem_bogovski} (i.e. by \eqref{bogovskii_formula} with $\Omega \coloneqq G_l$, $f\coloneqq \div ( \psi_l \phi u )$) to obtain $v_l \in C_0^\infty (G_l)$ such that $\| v_l \|_{W^{k,p} }\lec \| \psi_l \phi u \|_{W^{k,p}}$ for all $k\geq 0$, $p\in (1,\infty)$, where the implicit constant can be chosen independent of $l,R$. 

This gives that
\[
u_c \coloneqq \sum_{l=1}^L v_l
\]
belongs to $C_0^\infty (B_R \setminus B_{R-1})$ and
\[
\| u_c \|_{W^{k,p}}^p \lec \sum_{l=1}^L  \| v_l \|_{W^{k,p} }^p \lec \sum_{l=1}^L  \| \psi_l \phi u  \|_{W^{k,p}}^p \lec_p \| u \|_{W^{k,p} (B_R \setminus B_{R-1} ) }^p 
\] 
for $p\in (1,\infty )$. Thus 
\[
\begin{split}
\| u_c \|_{W^{1,p}}&\lec \| u \|_{W^{1,p} (B_R \setminus B_{R-1})} \lec_{a,p} \cC R^{-a} \qquad \text{ for } a\in [0,3/p'), p\in (1,\infty )\\
\| \Delta u_c \|_2 &\lec  \| u \|_{H^2 (B_R \setminus B_{R-1})} \lec_a \cC R^{-a} \qquad \text{ for } a\in [0,3/2),\\
\text{ and }\| \p_t u_c \|_{2}&\lec  \| \p_t u \|_{L^2 (B_R \setminus B_{R-1})} \lec_a \cC^2 R^{-a}\qquad \text{ for } a\in [0,5/2),
\end{split}
\]
where we used \eqref{lp_decay_of_derivatives}, \eqref{decay_of_ut} and the last inequality follows from the form of \eqref{bogovskii_formula}, which allows differentiation inside the integral. 

This and the embedding $H^2 \subset L^\infty$ gives \eqref{bound_on_nabla_uc}. Moreover
\[ \begin{split}
\| F_2 \|_2 &= \| (u_c \cdot \na )(u\phi ) + (u\phi  \cdot \na )v_c + (u_c \cdot \na ) u_c - \Delta u_c +\p_t u_c \|_2 \\
&\lec\| u_c \|_{W^{1,4}} \| u \|_{W^{1,4} (B_{R-1}^c)} +\| u_c \|_{W^{1,4}}^2 + \| \Delta u_c \|_2 + \| \p_t u_c \|_2 \\
&\lec_a \cC^2 R^{-a}
\end{split}\]
for $a\in [0,3/2)$, where we also used \eqref{lp_decay_of_derivatives}. This gives \eqref{decay_of_F2}, as required. 

\subsection{The Navier-Stokes equations on $B_R$ with small forcing}\label{sec_NSE_with_forcing}

In this section we discuss well-posedness of \eqref{NSE_general_form} with small small forcing $F$, and we prove \eqref{smallness_of_nabla_v}. 
\begin{lemma}[Strong solution to \eqref{NSE_general_form} for small forcing]\label{lem_solving_nse_with_forcing}
Let $T>0$ and $r : [0,T]\to D(A)$, and suppose that there exists $N\geq 1$ such that $\| \nabla  r (t) \|_{L^2(B_R)}+\| \nabla  r (t) \|_{L^4(B_R)}+\| D^2  r (t) \|_{L^2(B_R)} \leq N$ for all $t\in [0,T]$. Then there exists a unique strong solution $v$ to \eqref{NSE_general_form} above if $\| F  \|_{L^\infty ((0,T);L^2 (B_R))} \leq \varepsilon $ for some
\[ \varepsilon \in \left( 0 , C N^{3} \ee^{-2N^4C^2T}\right) , \]
where $C>1$ is a universal constant. Moreover,
\eqnb\label{strong_sol_smallness_toshow}
\| \na v (t) \|_{L^2 (B_R)} \leq \frac{\varepsilon }{N^2C } \left( \ee^{4N^4 C^2t }-1 \right)^\frac12 , \qquad \| \nabla \overline{\pi } \|_{L^p ((0,t); L^2 (B_R)} \lec_p  t^{\frac1p} N \varepsilon \ee^{2N^4C^2t }  
\eqne
for all $t\in [0,T]$ and $p\in (1,\infty )$.
\end{lemma}
Recall \eqref{stokes_op_def} for the definition of the Stokes operator $A$, and note that the assumption on $r$ implies that
\eqnb\label{to_use_r} 
\| r (t) \|_{L^\infty (B_R)} \leq C \| \nabla r \|_{L^2 (B_R)}^{\frac12} \| D^2 r \|_{L^2 (B_R)}^{\frac12} ,
\eqne
due to \eqref{agmon} and \eqref{homo_stokes_op}.

We note that taking $\varepsilon \coloneqq C(a) \cC^2 R^{-a}$, and $R$ as in \eqref{how_large_R}, the lemma implies \eqref{smallness_of_nabla_v}, as required.
 
The lemma can be proved using a standard Galerkin procedure, and we provide a sketch of the proof (inspired by \cite[Theorem~4.4 and Theorem~6.8]{NSE_book}) to keep track of the quantitative estimates. 
\begin{proof}
We first note that uniqueness follows in the same way as uniqueness of local-in-time strong solutions to the homogeneous Navier-Stokes equations (see Theorem 6.10 in \cite{NSE_book}, for example).

For existence, let $\mathcal{N} \coloneqq \mathrm{span} \,\{ a_1, \ldots , a_n \}$ denote the linear space spanned by the first $n$ eigenvalues of the Stokes operator $A$ (recall \eqref{stokes_op_def}) on $B_R$, that is for all $k$ $a_k\in D(A)$ and $A a_k = \lambda_k a_k$ for some $\lambda_k >0$ such that $0 < \lambda_k \leq \lambda_{k+1}$. 

We first show that for each $n$ there exist $c_1, c_2, \ldots , c_n \in C^1 ([0,T])$ such that
\[
w \coloneqq \sum_{k=1}^n c_k (t) a_k \in \mathcal{N}
\]
is a weak solution of the Galerkin approximation of 
\eqnb\label{galerkin_for_a}
\begin{split}
\p_t w + Aw + P_n \left( (w\cdot \nabla ) w \right) &=P_n F -P_n \left( (w\cdot \nabla ) v + (v\cdot \nabla ) w \right) \\
\div w &=0,\\
w(0) &= 0,
\end{split}
\eqne
where $P_n\colon L^2 \to \mathcal{N} \subset L^2$ is the orthogonal projection onto $\mathcal{N}$, i.e. that $w\in L^\infty ((0,T);L^2)\cap L^2 ((0,T);H^1)$ satisfies \eqref{distr_form_of_eqs} for $\phi \in \mathcal{N}$.

Indeed taking the inner product of the above equation with $a_k$ ($k=1,\ldots ,n$) we have
\[
c_k' + \sum_{j=1}^n c_j \int A a_ja_k + \sum_{i,j =1}^n c_i c_j \int (a_i \cdot \nabla )a_j\cdot  a_k = - \sum_{j=1}^n c_j \int \left( (a_j \cdot \nabla )r\cdot  a_k + (r\cdot \nabla ) a_j\cdot  a_k \right) + \int F \cdot a_k,
\]
and so using the facts that $A a_j = \lambda_j a_j$, that $a_j$'s are orthonormal in $L^2$ (see \cite[Theorem 2.24]{NSE_book}) and setting
\[
B_{ij}^{(k)} \coloneqq \int (a_i \cdot \nabla )a_j \cdot a_k,\qquad D_j^{(k)}\coloneqq \int \left( (a_j \cdot \nabla )r\cdot  a_k + (r\cdot \nabla ) a_j \cdot a_k \right), \qquad C^{(k)} \coloneqq \int F\cdot  a_k,
\]
we obtain a system of $n$ differential equations for $c_1, \ldots , c_n$,
\[
c_k'=- \sum_{i,j=1}^n c_i c_j B_{ij}^{(k)} - \sum_{j=1}^n (c_j D_j^{(k)} + \lambda_k )  + C^{(k)},
\]
with initial conditions $c_k(0) =0$ for $k=1,\ldots , n$. Since the right-hand side is locally Lipschitz, we obtain local in time well-posedness of the system (see Hartman \cite{hartman}). That the $c_k$'s exist for all times can be observed by testing \eqref{galerkin_for_a} by $w\in \mathcal{N}$, which gives that
\[
\frac12 \frac{\d }{\d t} \| w \|^2 +\| \na w \|^2 = -\int ( (w\cdot \nabla )r )\cdot w -\int F w\leq \frac12 \| \na w \|^2 + c   \| w \|^2 (1 + \| r \|^2_\infty ) + \frac12 \| F \|^2,
\]
where we used the cancellations $\int ((w\cdot \nabla ) w)\cdot w = \int ((r \cdot \nabla )w )\cdot w =0$, as well as integrated the term $\int ( (w\cdot \nabla )r )\cdot w$ by parts and applied Young's inequality. For brevity, we also used the notation $\| \cdot \| \equiv \| \cdot \|_{L^2 (B_R)}$ and $\| \cdot \|_p\equiv \| \cdot \|_{L^p (B_R)}$, which we continue for the rest of the proof.

The Gronwall inequality gives that 
\eqnb\label{est_on_w_L2} \sum_{k=1}^n c_k^2 = \| w \|^2 \leq \int_0^t \| F (s) \|^2  \ee^{cN^2 (t-s)} \d s \leq \varepsilon^2 \ee^{cN^2 t}  \eqne
for $t\geq 0$, which shows global existence of $c_k$'s, and also implies that $\int_0^T \| \nabla w \|^2 \lec \int_0^T \| F (t) \|^2 \d t + N^2 \int_0^T \| F (t) \|^2 \ee^{cN^2t} \d t <\infty $. 

Moreover $w$ is bounded in $L^\infty ((0,T);V)$ and in $L^2 ((0,T); H^2)$, uniformly in $n$. Indeed, multiplying the equation by $Aw$ we obtain
\[\begin{split}
\frac12 \frac{\d }{\d t } \| \na w \|^2 + \| Aw \|^2 &= \int\left( ((w\cdot \nabla ) w )Aw - \PP F Aw -  ((w\cdot \nabla )r) Aw - ((r\cdot \nabla ) w Aw \right) \\
& \leq \| w \|_\infty \| \nabla w \| \| Aw \|+ \| F \| \| Aw\|+ \| w \|_\infty \| \nabla r \| \| Aw \| + \| r \|_\infty \| \nabla w \| \| Aw \| \\
&\lec  \| \nabla w \|^{\frac32}  \| Aw \|^{\frac32} + \| F \| \| Aw \| + \| \nabla r \| \| \nabla w \|^{\frac12} \| Aw \|^{\frac32}  + \| r \|_\infty \| \nabla w \| \| Aw \| ,
\end{split}
\]
where we used \eqref{agmon} in the third inequality. Thus using Young's inequality we can absorb $\| Aw \|^2$ on the left-hand side to obtain
\[\begin{split}
 \frac{\d }{\d t } \| \na w \|^2 + \| A w \|^2 &\leq C^2 \| \na w \|^6  + C^2 \| \na w \|^2 ( \| \na r \|^4 + \| r \|_\infty^2 ) + \| F \|^2\\
 &\leq C^2 \| \na w \|^6  + C^2 N^4 \| \na w \|^2  + \varepsilon^2
\end{split}
\]
for some $C>\max\{ 1,c\}$, where $c$ is from \eqref{est_on_w_L2}. Thus, since 
\[ g(t) \coloneqq  \frac{\varepsilon^2 }{N^4C^2 } \left( \ee^{4N^4C^2t }-1 \right) 
\]
satisfies 
\[
g'(t) \geq C^2 g(t)^3  +  C^2N^4 \,g(t) + \varepsilon^2
\]
for $t\in [0,T]$, we have that 
\eqnb\label{est_on_na_w_L2}
\| \na w \|^2 \leq g(t) \qquad \text{ for }t\in [0,T],
\eqne
which also implies that $\| D^2 w \|_{L^2 ((0,T)\times B_R )}^2 \lec \| A w \|_{L^2 ((0,T)\times B_R )}^2 \leq g(T) - g(0) =  {\varepsilon }{(NC)^{-1} } \left( \ee^{4NCT }-1 \right) $, where we also used \eqref{homo_stokes_op}.\\

Finally \eqref{galerkin_for_a} shows that $\| \p_t w \|_{L^{\frac43}((0,T);V^*)} $ is bounded uniformly in $n$ (recall that $w\equiv w_n $ is the Galerkin approximation for given $n$), which can be shown by a standard argument, using H\"older's inequality, Lebesgue interpolation and Sobolev embedding $H^1_0\subset L^6$, see for example \cite[Theorem 4.4, Step 3]{NSE_book}.

This estimate on the time derivative lets us use the Aubin-Lions lemma (see \cite{aubin,lions} or \cite[Theorem~2.1 in Chapter~III]{temam_book}) to extract a subsequence $\{ w_{n_k} \}$ such that 
\[
\begin{split}
w_{n_k} \to& v \qquad \text{ in } L^3 ((0,T)\times B_R),\\
w_{n_k} \stackrel{*}{\rightharpoonup} &v \qquad \text{ in } L^\infty ((0,T); V),\\
D^2 w_{n_k } \rightharpoonup &D^2 v \qquad \text{ in } L^2 ((0,T)\times B_R)
\end{split}
\]
for some $v\in L^\infty ((0,T);V)\cap L^2 ((0,T);H^2)$.
This mode of convergence enables us to take the limit in the weak formulation of the equation for $w_{n_k}$, and so shows that $v$ is the required solution. 

The estimate for $\nabla v$ in \eqref{strong_sol_smallness_toshow} follows from \eqref{est_on_na_w_L2}. As for the estimate for $\nabla \overline{\pi }$ in \eqref{strong_sol_smallness_toshow} we have
\[\begin{split}
\| (v\cdot \nabla ) r + (r\cdot \nabla ) v \|_2 &\leq \| v \|_4 \| \nabla r \|_4 + \| r\|_{\infty } \| \nabla v \|_2 \lec N\left( \| v \|_2^{\frac14} \| \nabla v \|_2^{\frac34}  + \| \nabla v \|_2 \right) \lec N \varepsilon \ee^{2N^4C^2t }
\end{split}
\]
at each time, where we used \eqref{to_use_r} and \eqref{interp1} in the second inequality, as well as \eqref{est_on_w_L2} and \eqref{est_on_na_w_L2} in the last. Thus  \eqref{svw_ineq} gives
\[\begin{split}
\| \nabla \overline{\pi } \|_{L^p ((0,t); L^2 (B_R)} &\lec_p  \| F- (v\cdot \nabla )r - (r\cdot \nabla )v \|_{L^p ((0,t); L^2 (B_R))} \lec t^{\frac1p} N \varepsilon \ee^{2N^4C^2t }  
\end{split}
\]
for every $p\in (1,\infty )$, as required.
\end{proof}

\section*{Acknowledgement} The author has been partially supported by the Simons Foundation. 

\bibliography{literature}{}

\begin{thebibliography}{10}

\bibitem{aubin}
J.-P. Aubin.
\newblock Un th\'{e}or\`eme de compacit\'{e}.
\newblock {\em C. R. Acad. Sci. Paris}, 256:5042--5044, 1963.

\bibitem{beckner}
W.~Beckner.
\newblock Inequalities in {F}ourier analysis.
\newblock {\em Ann. of Math. (2)}, 102(1):159--182, 1975.

\bibitem{bogovskii_79}
M.~E. Bogovski\u{\i}.
\newblock Solution of the first boundary value problem for an equation of
  continuity of an incompressible medium.
\newblock {\em Dokl. Akad. Nauk SSSR}, 248(5):1037--1040, 1979.

\bibitem{bogovskii_80}
M.~E. Bogovski\u{\i}.
\newblock Solutions of some problems of vector analysis, associated with the
  operators {${\rm div}$} and {${\rm grad}$}.
\newblock In {\em Theory of cubature formulas and the application of functional
  analysis to problems of mathematical physics}, volume 1980 of {\em Trudy Sem.
  S. L. Soboleva, No. 1}, pages 5--40, 149. Akad. Nauk SSSR Sibirsk. Otdel.,
  Inst. Mat., Novosibirsk, 1980.

\bibitem{borchers_sohr}
W.~Borchers and H.~Sohr.
\newblock On the equations {${\rm rot}\,{\bf v}={\bf g}$} and {${\rm div}\,{\bf
  u}=f$} with zero boundary conditions.
\newblock {\em Hokkaido Math. J.}, 19(1):67--87, 1990.

\bibitem{ckn_84}
L.~Caffarelli, R.~Kohn, and L.~Nirenberg.
\newblock First order interpolation inequalities with weights.
\newblock {\em Compositio Math.}, 53(3):259--275, 1984.

\bibitem{chae_2002}
D.~Chae.
\newblock On the well-posedness of the {E}uler equations in the
  {T}riebel-{L}izorkin spaces.
\newblock {\em Comm. Pure Appl. Math.}, 55(5):654--678, 2002.

\bibitem{chae_wolf}
D.~Chae and J.~Wolf.
\newblock On the {S}errin-type condition on one velocity component for the
  {N}avier-{S}tokes equations.
\newblock 2019.
\newblock arXiv:1911.02699.

\bibitem{CCRT_07}
S.~I. Chernyshenko, P.~Constantin, J.~C. Robinson, and E.~S. Titi.
\newblock A posteriori regularity of the three-dimensional {N}avier-{S}tokes
  equations from numerical computations.
\newblock {\em J. Math. Phys.}, 48(6):065204, 15, 2007.

\bibitem{constantin_1986}
P.~Constantin.
\newblock Note on loss of regularity for solutions of the {$3$}-{D}
  incompressible {E}uler and related equations.
\newblock {\em Comm. Math. Phys.}, 104(2):311--326, 1986.

\bibitem{cf_vorticity}
P.~Constantin and C.~Fefferman.
\newblock Direction of vorticity and the problem of global regularity for the
  {N}avier-{S}tokes equations.
\newblock {\em Indiana Univ. Math. J.}, 42(3):775--789, 1993.

\bibitem{constantin_foias}
P.~Constantin and C.~Foias.
\newblock {\em Navier-{S}tokes equations}.
\newblock Chicago Lectures in Mathematics. University of Chicago Press,
  Chicago, IL, 1988.

\bibitem{galdi_book}
G.~P. Galdi.
\newblock {\em An introduction to the mathematical theory of the
  {N}avier-{S}tokes equations}.
\newblock Springer Monographs in Mathematics. Springer, New York, second
  edition, 2011.
\newblock Steady-state problems.

\bibitem{gallagher_97}
I.~Gallagher.
\newblock The tridimensional {N}avier-{S}tokes equations with almost
  bidimensional data: stability, uniqueness, and life span.
\newblock {\em Internat. Math. Res. Notices}, (18):919--935, 1997.

\bibitem{gk_03}
Z.~Gruji\'{c} and I.~Kukavica.
\newblock A remark on time-analyticity for the {K}uramoto-{S}ivashinsky
  equation.
\newblock {\em Nonlinear Anal.}, 52(1):69--78, 2003.

\bibitem{hartman}
P.~Hartman.
\newblock {\em Ordinary differential equations}.
\newblock S. M. Hartman, Baltimore, Md., 1973.
\newblock Corrected reprint.

\bibitem{heywood_88}
J.~G. Heywood.
\newblock Epochs of regularity for weak solutions of the {N}avier-{S}tokes
  equations in unbounded domains.
\newblock {\em Tohoku Math. J. (2)}, 40(2):293--313, 1988.

\bibitem{hopf_1951}
E.~Hopf.
\newblock {\"Uber die Anfangswertaufgabe f\"ur die hydrodynamischen
  Grundgleichungen}.
\newblock {\em Math. Nachr.}, 4:213--231, 1951.
\newblock {(An English translation due to Andreas Kl\"ockner is available at
  \emph{http://www.dam.brown.edu/people/menon/publications/notes/hopf-NS.pdf}.)}.

\bibitem{kelliher_08}
J.~P. Kelliher.
\newblock Expanding domain limit for incompressible fluids in the plane.
\newblock {\em Comm. Math. Phys.}, 278(3):753--773, 2008.

\bibitem{kerr_2018}
R.~M. Kerr.
\newblock Enstrophy and circulation scaling for {N}avier-{S}tokes reconnection.
\newblock {\em J. Fluid Mech.}, 839:R2, 14, 2018.

\bibitem{k_01}
I.~Kukavica.
\newblock Space-time decay for solutions of the {N}avier-{S}tokes equations.
\newblock volume~50, pages 205--222. 2001.
\newblock Dedicated to Professors Ciprian Foias and Roger Temam (Bloomington,
  IN, 2000).

\bibitem{ko_d3u}
I.~Kukavica and W.~S. O\.za\'nski.
\newblock An anisotropic regularity condition for the 3{D} incompressible
  {N}avier–{S}tokes equations for the entire exponent range.
\newblock {\em Appl. Math. Lett.}, 2021.
\newblock to appear.

\bibitem{kt_06}
I.~Kukavica and J.~J. Torres.
\newblock Weighted bounds for the velocity and the vorticity for the
  {N}avier-{S}tokes equations.
\newblock {\em Nonlinearity}, 19(2):293--303, 2006.

\bibitem{kt_07}
I.~Kukavica and J.~J. Torres.
\newblock Weighted {$L^p$} decay for solutions of the {N}avier-{S}tokes
  equations.
\newblock {\em Comm. Partial Differential Equations}, 32(4-6):819--831, 2007.

\bibitem{ladyzhenskaya}
O.~A. Lady\v{z}enskaja.
\newblock Uniqueness and smoothness of generalized solutions of
  {N}avier-{S}tokes equations.
\newblock {\em Zap. Nau\v{c}n. Sem. Leningrad. Otdel. Mat. Inst. Steklov.
  (LOMI)}, 5:169--185, 1967.

\bibitem{leray_34}
J.~Leray.
\newblock {Sur le mouvement d'un liquide visqueux emplissant l'espace}.
\newblock {\em Acta Math.}, 63:193--248, 1934.
\newblock (An English translation due to Robert Terrell is available at
  \emph{http://www.math.cornell.edu/~bterrell/leray.pdf} and
  \emph{https://arxiv.org/abs/1604.02484}.).

\bibitem{lions}
J.-L. Lions.
\newblock {\em Quelques m\'{e}thodes de r\'{e}solution des probl\`emes aux
  limites non lin\'{e}aires}.
\newblock Dunod; Gauthier-Villars, Paris, 1969.

\bibitem{np1}
J.~Neustupa and P.~Penel.
\newblock Regularity of a suitable weak solution to the {N}avier-{S}tokes
  equations as a consequence of regularity of one velocity component.
\newblock In {\em Applied nonlinear analysis}, pages 391--402. Kluwer/Plenum,
  New York, 1999.

\bibitem{op}
W.~S. O\.{z}a\'{n}ski and B.~C. Pooley.
\newblock Leray's fundamental work on the {N}avier-{S}tokes equations: a modern
  review of {\it ``{s}ur le mouvement d'un liquide visqueux emplissant
  l'espace''}.
\newblock In {\em Partial differential equations in fluid mechanics}, volume
  452 of {\em London Math. Soc. Lecture Note Ser.}, pages 113--203. Cambridge
  Univ. Press, Cambridge, 2018.

\bibitem{prodi}
G.~Prodi.
\newblock Un teorema di unicit\`a per le equazioni di {N}avier-{S}tokes.
\newblock {\em Ann. Mat. Pura Appl. (4)}, 48:173--182, 1959.

\bibitem{raugel_sell_93}
G.~Raugel and G.~R. Sell.
\newblock Navier-{S}tokes equations on thin {$3$}{D} domains. {I}. {G}lobal
  attractors and global regularity of solutions.
\newblock {\em J. Amer. Math. Soc.}, 6(3):503--568, 1993.

\bibitem{robinson_torus}
J.~C. Robinson.
\newblock Using periodic boundary conditions to approximate the
  {N}avier–{S}tokes equations on $\mathbb{R}^3$ and the transfer of
  regularity.
\newblock 2020.
\newblock arXiv:2008.04725.

\bibitem{NSE_book}
J.~C. Robinson, J.~L. Rodrigo, and W.~Sadowski.
\newblock {\em The three-dimensional {N}avier-{S}tokes equations}, volume 157
  of {\em Cambridge Studies in Advanced Mathematics}.
\newblock Cambridge University Press, Cambridge, 2016.
\newblock Classical theory.

\bibitem{serrin}
J.~Serrin.
\newblock The initial value problem for the {N}avier-{S}tokes equations.
\newblock In {\em Nonlinear {P}roblems ({P}roc. {S}ympos., {M}adison, {W}is.,
  1962)}, pages 69--98. Univ. of Wisconsin Press, Madison, Wis., 1963.

\bibitem{skalak}
Z.~Skalak.
\newblock The end-point regularity criterion for the {N}avier-{S}tokes
  equations in terms of {$\partial_3 u$}.
\newblock {\em Nonlinear Anal. Real World Appl.}, 55:103120, 10, 2020.

\bibitem{svw_86}
H.~Sohr and W.~von Wahl.
\newblock On the regularity of the pressure of weak solutions of
  {N}avier-{S}tokes equations.
\newblock {\em Arch. Math. (Basel)}, 46(5):428--439, 1986.

\bibitem{stein_singular_int}
E.~M. Stein.
\newblock {\em Singular integrals and differentiability properties of
  functions}.
\newblock Princeton Mathematical Series, No. 30. Princeton University Press,
  Princeton, N.J., 1970.

\bibitem{tao_2013}
T.~Tao.
\newblock Localisation and compactness properties of the {N}avier-{S}tokes
  global regularity problem.
\newblock {\em Anal. PDE}, 6(1):25--107, 2013.

\bibitem{temam_book}
R.~Temam.
\newblock {\em Navier-{S}tokes equations}, volume~2 of {\em Studies in
  Mathematics and its Applications}.
\newblock North-Holland Publishing Co., Amsterdam-New York, revised edition,
  1979.
\newblock Theory and numerical analysis, With an appendix by F. Thomasset.

\bibitem{wwz}
W.~Wang, D.~Wu, and Z.~Zhang.
\newblock Scaling invariant {S}errin criterion via one velocity component for
  the {N}avier-{S}tokes equations.
\newblock 2020.
\newblock arXiv:2005.11906.

\end{thebibliography}
\end{document}